\documentclass[11pt, a4paper]{amsart}

\usepackage{amsfonts,amsmath,amssymb,amscd}

\newtheorem{theorem}{Theorem}[section]
\newtheorem{lemma}[theorem]{Lemma}
\newtheorem{proposition}[theorem]{Proposition}
\newtheorem{corollary}[theorem]{Corollary}
\newtheorem{claim}[theorem]{Claim}

\theoremstyle{definition}
\newtheorem{definition}[theorem]{Definition}
\newtheorem{remark}[theorem]{Remark}
\newtheorem{example}[theorem]{Example}

\numberwithin{equation}{section}

\title[$\times_R$-Bialgebras]
{$\times_R$-Bialgebras associated with iterative $q$-difference rings}

\author[A.~Masuoka]{Akira Masuoka}
\address{Akira Masuoka: 
Institute of Mathematics, 
University of Tsukuba, 
Ibaraki 305-8571, Japan}
\email{akira@math.tsukuba.ac.jp}

\author[M.~Yanagawa]{Makoto Yanagawa}
\address{Makoto Yanagawa: 
Graduate School of Pure and Applied Sciences, 
University of Tsukuba, 
Ibaraki 305-8571, Japan}
\email{myanagawa@math.tsukuba.ac.jp}

\begin{document}

\begin{abstract}
Realizing the possibility suggested by Hardouin \cite{H}, we show that her own Picard-Vessiot Theory
for iterative $q$-difference rings is covered by the (consequently, more general) framework, settled 
by Amano and Masuoka \cite{AM}, of
artinian simple module algebras over a cocommutative pointed Hopf algebra. 
An essential point is to represent iterative $q$-difference modules over an iterative $q$-difference 
ring $R$, by modules over a certain cocommutative $\times_R$-bialgebra. Recall that the notion of  
$\times_R$-bialgebras was defined by Sweedler \cite{S3}, as a generalization of bialgebras. 
\end{abstract}

\maketitle



\section*{Introduction}\label{sec:intro}

Picard-Vessiot Theory is Galois Theory for linear differential equations; see van der Put and Singer \cite{PS2} (2003).
There were recognized some analogous theories in which a single differential operator is replaced by 
a family of such operators or iterative differential operators, for example.  
Takeuchi \cite{T} (1989) reconstructed and unified those theories by characteristic-free, Hopf algebraic 
approach. 
Amano and Masuoka \cite{AM} (2005) extended 
Takeuchi's theory, to absorb as well Picard-Vessiot Theory for difference
equations such as developed by van der Put and Singer \cite{PS1} (1997). 
Such a unification of differential and difference Picard-Vessiot Theories was 
done earlier by Andr\'{e} \cite{An} (2001) from the different standpoint of non-commutative geometry. 
In the framework of \cite{AM}, 
(i)~differential or difference operators, (ii)~differential fields or difference total rings, (iii)~differential or difference
equations, and (iv)~linear algebraic groups of differential or difference automorphisms, all
in the classical situation are replaced by (i)~actions by an appropriately chosen, cocommutative pointed
Hopf algebra $D$, (ii)~artinian simple or AS $D$-module algebras $R$,  (iii)~modules over the smash-product algebra
$R\# D$, and (iv)~affine group schemes (or equivalently, commutative Hopf algebras) of $D$-module
algebra automorphisms, respectively. 

One feature of the Hopf-algebraic approach is first to define Picard-Vessiot extensions abstractly,
and then to characterize them as minimal splitting fields or algebras of equations (or of
module objects); recall that such fields or algebras are chosen as the definition of Picard-Vessiot
extensions by the classical approach. 
As a benefit, the Galois correspondence turns to be just $``$a dictionary between Hopf ideals and
intermediate artinian simple $D$-modules algebras," as was expressed by Bertrand \cite{B} (2011). 

Hardouin \cite{H} (2010) developed Picard-Vessiot Theory for iterative $q$-difference operators, 
using not results, but ideas from Matzat and van der Put \cite{MaP} (2003), 
who developed, probably without knowing \cite{T}, the theory for iterative differential operators.  
In the introduction of the paper, Hardouin says
$``$This analogy between iterative differential
Galois theory and iterative difference Galois theory could perhaps be explained in a
more theoretical way, as it is done in the paper of Y. Andr\'{e} \cite{An} for classical theories."
The objective of this paper is to realize this suggestion. 
In fact, we will show that the main results of \cite{H}, which will be reproduced as Theorems 
\ref{thm:HardouinThm4.7}--\ref{thm:HardouinThm4.20}, 
follow from results of \cite{AM}; this is formulated as Claim \ref{claim}. 
Suppose that $R$ is a commutative ring
which includes the field $C(t)$ of rational functions over a field $C$. 
Given an element $q \in C\setminus \{0,1\}$ and an automorphism $\sigma_q$ on $R$ which extends
the $q$-difference operator $f(t) \mapsto f(qt)$ on $C(t)$, 
the ring $R$ is called an \emph{iterative} $q$-\emph{difference ring}
if it is given an \emph{iterative} $q$-\emph{difference operator}, i.e., an $\infty$-sequence 
$\delta^{(0)}_R = \mathrm{id}_R, \delta^{(1)}_R, \delta^{(2)}_R, \dots$ of operators 
which satisfy some conditions that include
\[ \delta^{(k)}_R( x y ) = \sum_{ i + j = k } \sigma_q^i \circ \delta^{(j)}_R(x) \, \delta^{(i)}_R( y ),\ x, y \in R. \]
This last condition is obviously equivalent to
\[ \delta^{(k)}_R( x y ) = \sum_{ i + j = k } \delta^{(i)}_R( x )\, \sigma_q^i \circ \delta^{(j)}_R(y),\ x, y \in R. \]
An essential point for us is to refine this equivalence into the cocommutativity of a certain $\times_R$-bialgebra,
say $\mathcal{H}$; recall that the notion of  
$\times_R$-bialgebras was defined by Sweedler \cite{S3} (1974), as a generalization of bialgebras.
The module objects, i.e., iterative $q$-difference modules, over $R$ 
are identified with $\mathcal{H}$-modules. 
Consequently, operations, such as tensor products and taking duals, on iterative $q$-difference modules
can be well controlled by structures, such as the coproduct and an analogue of antipodes, on $\mathcal{H}$. 
As a final step of proving Claim \ref{claim} in characteristic zero, $\mathcal{H}$-modules are identified with 
$R\# D$-modules for an appropriate cocommutative pointed Hopf algebra $D$. In positive characteristic,
the $\mathcal{H}$-modules are embedded into the category of $R \# D$-modules for a
distinct $D$. 

The contents of this paper are as follows. 
Section \ref{sec:review} reproduces some results from \cite{AM} that will be needed later, 
partially in reformulated form. Section \ref{sect:xR} is devoted to preliminaries on 
$\times_R$-bialgebras. In Section \ref{sect:iter-dif-ring}, we associate to each iterative $q$-difference
ring $R$, a cocommutative $\times_R$-bialgebra $\mathcal{H}$ such as explained above; see Theorem \ref{thm:H}. 
This $\mathcal{H}$ is naturally characterized in the endomorphism algebra $\operatorname{End}(R)$; see
Proposition \ref{prop:injective_representation}.
In Section \ref{sec:proof_of_claim}, we prove the desired Claim \ref{claim}. The argument
will show that some additional results on AS $D$-module algebras, such as given in \cite{AM, A}, 
as well 
can apply to iterative $q$-difference rings
in a generalized (and hopefully, more natural) situation; see Section \ref{subsec:added_remark}.    
The final Section \ref{sec:remarks} gives the remark that the results on 
iterative $q$-difference rings shown in Section \ref{sec:proof_of_claim} are directly generalized to 
rings given a $q$-\emph{skew iterative} $\sigma$-\emph{derivation} such as Heiderich \cite{He} (2010) defines. 

As a remarkable new direction of relevant research, 
Saito and Umemura \cite{SaU} (preprint, 2012) explore a quantized world 
associated with non-linear differential-difference equations including those defined by 
iterative $q$-difference operators.

\section{Quick review on Picard-Vessiot theory of\\ artinian simple module algebras}\label{sec:review}

In this section we work over a fixed field $\Bbbk$. Let $D$ be a 
Hopf algebra with coproduct $\Delta : D \to D \otimes D$ and counit $\varepsilon : D \to \Bbbk$.  
For this and any other coproducts
we use the sigma notation \cite[Section 1.2]{S2}
\[\Delta (d) = \sum \, d_1 \otimes d_2.\] 
We assume that $D$ is cocommutative and pointed. 
Thus, $D$ equals a smash product $D^1 \# kG$ of a cocommutative irreducible Hopf algebra $D^1$ by a
group Hopf algebra $\Bbbk G$; see \cite[Section 8.1]{S2}. 
If the characteristic $\operatorname{char}\Bbbk$ of $\Bbbk$ is zero,
then $D^1$ equals the universal envelope of the Lie algebra consisting of all primitives in $D$. 
We assume in addition
\begin{equation}\label{BW_assumption} 
D^1\ \text{is \emph{Birkhoff-Witt} as a coalgebra}. 
\end{equation}
This means that $D^1$ is the tensor product of (possibly, infinitely many) copies of
the coalgebra spanned by an $\infty$-divided power sequence. 
This is always satisfied if $\operatorname{char}\Bbbk= 0$. 
If $\operatorname{char}\Bbbk = p > 0$, the assumption is equivalent to saying that the Verschiebung
map $D^1 \to D^1 \otimes \Bbbk^{1/p}$ is surjective. 

Recall from \cite[p.153]{S2} 
the definition of $D$-\emph{module algebras}, by which we mean left $D$-module
algebras that are non-zero and commutative; the action will be written as 
$d \rightharpoonup$, $d \in D$. An algebra given a family of differential operators (in characteristic zero),
iterative differential operators or/and a family of inversive difference operators are presented as a $D$-module
algebra for an appropriate $D$; see \cite[Introduction]{AM}.  

Given a $D$-module algebra $A$, the subalgebra of $D$-invariants in $A$ is given by
\[ A^D = \{ a \in A \mid d\rightharpoonup  a = \varepsilon(d)a,\ d \in D \}. \]
If $A$ is simple, i.e., contains no non-trivial $D$-stable ideal, then $A^D$ is a field. 
A $D$-module algebra 
is said to be \emph{artinian simple} or \emph{AS} \cite[Definition 2.6]{AM}, if it is artinian (as a commutative ring)
and simple. An AS $D$-module algebra is the direct product
of mutually isomorphic, finitely many fields on which the group $G$ in $D$ acts transitively, 
whence it is \emph{total}, 
i.e., every non-zero divisor is invertible \cite[Corollary 2.5]{AM}. 
In an AS $D$-module algebra, a $D$-module subalgebra is AS if and only if 
it is total \cite[Lemma 2.8]{AM}.

An inclusion $K \subset L$ of AS $D$-module algebras
is said to be a \emph{Picard-Vessiot} or \emph{PV extension} \cite[Definition 3.3]{AM}, 

(i)\ if their $D$-invariants coincide, $K^D = L^D$, and 

(ii)\ if there exists (necessarily, uniquely \cite[Proposition 3.4(iii)]{AM}) an intermediate $D$-module
algebra $K \subset A \subset L$ such that 
\begin{itemize}
\item[(a)] the $D$-invariants $H :=(A \otimes_K A)^D$ in
$A\otimes_K A$, on which $D$ acts diagonally, i.e., through $\Delta$, generates the left (or right) $A$-module
$A \otimes_K A$, and 
\item[(b)] the total quotient ring $Q(A)$ (i.e., the localization by all non-zero divisors) of $A$ coincides with $L$. 
\end{itemize}
Suppose that this is the case. According to traditional notation, the field 
$K^D \, (=L^D)$ will be denoted by $C$. Obviously, $A^D = C$.
As a $D$-module algebra, $A$ is simple \cite[Corollary 3.12]{AM}. Moreover, it contains 
all primitive idempotents in $L$, 
so that $A$ is the direct product $A_1\times \dots \times A_r$ of mutually isomorphic
integral domains $A_1,\dots A_r$, and
$L$ is the direct product $L_1 \times \dots \times L_r$, where $L_i$ is the quotient field of $A_i$. 
The map $\mu : A \otimes_C H \to A \otimes_K A$, $\mu(a \otimes h) = (a\otimes 1)\, h$ is necessarily bijective. 
The commutative $C$-algebra $H$ has a unique $C$-Hopf algebra structure such that 
\[ \theta : A \to A \otimes_C H, \ \theta(a) = \mu^{-1}(1 \otimes a) \]
makes $A$ into a right $H$-comodule (algebra) over $C$. 
We call $A$ (resp., $H$) the \emph{principal $D$-module algebra} (resp., the \emph{Hopf algebra}) for $L/K$, 
and often refer to the triple $(L/K, A, H)$ as a PV extension. 
We let $\mathbf{G}(L/K)= \mathrm{Spec}_C H$ denote the affine group scheme over $C$
which corresponds to $H$, and call it the \emph{PV group scheme} for $L/K$. 
This is naturally isomorphic to the group-valued functor $\mathbf{Aut}_{D,K\text{-}\mathrm{alg}}(A)$ which associates
to each commutative $C$-algebra $T$, the group of all $D$-linear $K \otimes_C T$-algebra automorphisms
on $A \otimes_C T$ \cite[Remark 3.11]{AM}. 
The affine scheme $\mathrm{Spec}_C A$ is a $\mathbf{G}(L/K)$-torsor over 
$\mathrm{Spec}_C K$, or in other words, $A/K$ is an $H$-Galois extension. This means that the $A$-algebra
map ${}_A\theta : A \otimes_K A \to A \otimes_C H$, ${}_A\theta(a \otimes b)=(a\otimes 1) \theta(b)$ 
is an isomorphism \cite[Proposition 3.4]{AM}.   

We remark that $\mathbf{G}(L/K)$ is not necessarily algebraic since we do not assume that
$L$ is finitely generated over $K$; see Lemma \ref{lem:finite_generation} below. 

\begin{theorem}[\textbf{Galois correspondence}--\cite{AM}, Theorem~3.9]\label{thm:Gal_corresp}
Let $(L/K, A, H)$ be a PV extension of AS $D$-module algebras. 
\begin{itemize}
\item[(1)] If $K \subset M \subset L$ is an intermediate AS (or equivalently, total) $D$-module algebra, 
then $L/M$ is a PV extension
which has $AM$ as its principal $D$-module algebra. The correspondence $M \mapsto \mathbf{G}(L/M)$ 
gives a bijection from the set of all intermediate AS $D$-module algebras $K \subset M \subset L$ to
the set of all closed subgroup schemes in $\mathbf{G}(L/K)$. 
\item[(2)] An intermediate AS $D$-module algebra $M$ is a PV extension over $K$ if and only if
the corresponding $\mathbf{G}(L/M)$ is normal in $\mathbf{G}(L/K)$. In this case, $\mathbf{G}(M/K)$
is naturally isomorphic to the quotient group sheaf $\mathbf{G}(L/K)\tilde{\tilde{/}}\mathbf{G}(L/M)$
in the fpqf topology.
\end{itemize}
\end{theorem}

\begin{remark}\label{rem:inverse}
In the situation of the theorem above, let $M$ be an intermediate AS $D$-module algebra in $L/K$. 
The closed subgroup scheme $\mathbf{G}(L/M)$ of $\mathbf{G}(L/K)$ is of the 
form $\mathrm{Spec}_C(H/\mathfrak{a})$, where $\mathfrak{a}$ is a Hopf ideal of $H$. 
\begin{itemize}
\item[(1)] (cf.~\cite[Lemma 4.18]{H})\ By \cite[Theorem 3.9(1)]{AM}, $M$ recovers from $\mathfrak{a}$ as the
set consisting of all elements $x \in L$ such that
\[ x \otimes 1 \equiv 1 \otimes x \mod{\mathfrak{a}\, (L \otimes_K L)}\quad \text{in} \quad L \otimes_K L. \]
Suppose that $x= a/b$, where $a, b \in A$ with $b$ a non-zero divisor. Then one sees that
the last condition is equivalent to
\[ (a\otimes 1) \, \theta(b) \equiv (b\otimes 1) \, \theta(a) \mod{A \otimes_C \mathfrak{a}}\quad \text{in}\quad A \otimes_C H. \]
\item[(2)] Assume that $M/K$ is a PV extension, or equivalently, $\mathbf{G}(L/M)$ is normal in $\mathbf{G}(L/K)$.  
Then one sees from \cite[p.756, line --12]{AM} that the principal $D$-module algebra for $M/K$ consists 
of all elements $a \in A$ such that 
\[ \theta(a) \equiv a \otimes 1 \, \mod{A \otimes_C \mathfrak{a}}\quad \text{in}\quad A \otimes_C H. \]
\end{itemize}
\end{remark}

We will see that the theorem and the remark above imply Parts 1, 2 of Theorem 4.20 of
\cite{H}, which will be reproduced as Theorem \ref{thm:HardouinThm4.20}.  For the remaining 
Part 3, we prove the following. 

\begin{proposition}\label{prop:separable}
Let $(L/K, A, H)$ be a PV extension of AS $D$-module algebras. Assume that
$K$ is a field, and the field $C$ is perfect. Then the following are equivalent:
\begin{itemize}
\item[(a)] $H$ is reduced;
\item[(b)] $A \otimes_K\widetilde{K}$ is reduced for any field extension $\widetilde{K}/K$; 
\item[(c)] $L \otimes_K\widetilde{K}$ is reduced for any field extension $\widetilde{K}/K$. 
\end{itemize}
\end{proposition} 
\begin{proof}
The equivalence (a) $\Leftrightarrow$ (b) follows from the $A$-algebra isomorphism ${}_A\theta :
A \otimes_K A \overset{\simeq}{\longrightarrow} A \otimes_C H$. Obviously, (c) $\Rightarrow$ (b). 
The converse holds true since $L \otimes_K\widetilde{K}$ is a localization of $A \otimes_K\widetilde{K}$
by non-zero divisors. 
\end{proof}

\begin{lemma}[\cite{AM}, Corollary~4.8]\label{lem:finite_generation}
Given a PV extension $(L/K, A, H)$ of AS $D$-module algebras, the following are equivalent:
\begin{itemize}
\item[(a)] $L$ is the smallest AS $D$-module subalgebra in $L$ that includes $K$ and some finitely many
elements in $L$; 
\item[(b)] $L$ is the total quotient ring of some finitely generated $K$-subalgebra of $L$;
\item[(c)] $A$ is finitely generated as a $K$-algebra;
\item[(d)] $H$ is finitely generated as a $C$-algebra, or in other words, $\mathbf{G}(L/K)$ is 
algebraic. 
\end{itemize}
\end{lemma}  

If the equivalent conditions above are satisfied, we say that the PV extension is \emph{finitely generated}. 

Let $K$ be a $D$-module algebra. Then one constructs the smash-product algebra $K \# D$; recall from 
\cite[p.~153]{S2} that 
it is generated by $K$, $D$ subject to the relation $d x = \sum (d_1 \rightharpoonup x)d_2,\ d \in D, x\in K$. 
One sees that $K$ is a left $K \# D$-module by
\begin{equation}\label{module}
(x \# d) \rightharpoonup y := x(d \rightharpoonup y), \quad d \in D,\ x, y \in K.
\end{equation}
If $V$ is a left $K\# D$-module and if $L$ is a $D$-module algebra including $K$, 
then $L \otimes_K V$ is naturally a left $L \# D$-module, on which $D$ acts diagonally.  

Assume that $K$ is AS. Then every left $K \# D$-module $V$ is free as a left $K$-module \cite[Corollary 2.5]{AM}.
Assume that $V$ has a finite $K$-free rank, say $n$. 
Let $L$ be an AS $D$-module algebra including $K$.  We call $L$ 
a \emph{splitting algebra} for $V$ \cite[Definition 4.1]{AM}, if there is an isomorphism
$L \otimes_K V \simeq L^n$ of left $L\# D$-modules, where $L^n$ denotes the direct sum of $n$
copies of $L$. Choose a $K$-free basis $v_1,\dots,v_n$ of $V$, and set
$\boldsymbol{v} ={}^t(v_1,\dots , v_n)$. Then the $D$-module structure on $V$ is represented uniquely by 
$n\times n$ matrices $\mathbf{M}(d)$, $d \in D$, with entries in $K$, so that
\begin{equation}\label{equation}
d \rightharpoonup \boldsymbol{v} = \mathbf{M}(d)\, \boldsymbol{v}, \quad d \in D. 
\end{equation}
The set $\operatorname{Hom}_{K\# D}(V, L)$ of all $K \# D$-linear maps $V \to L$ naturally forms a
vector space over the field $L^D$, whose dimension is at most $n$. This is embedded into 
$L^n$ via 
$f \mapsto {}^t ( f(v_1), \dots, f(v_n) )$, where $f \in \operatorname{Hom}_{K \# D} (V, L)$. An 
element $\boldsymbol{x} =  {}^t ( x_1, \dots, x_n ) \in L^n$  is in $\operatorname{Hom}_{K \# D} (V, L)$
if and only if $\boldsymbol{x}$ is a solution of the equation \eqref{equation}, i.e., 
$d \rightharpoonup \boldsymbol{x} = \mathbf{M}(d)\, \boldsymbol{x}$, $d \in D$. Thus, 
$\operatorname{Hom}_{K \# D} (V, L)$ is the solution space for \eqref{equation}. By \cite[Lemma 4.2]{AM}, 
$L$ is a splitting algebra for $V$ if and only if the $L^D$-dimension of 
$\operatorname{Hom}_{K \# D} (V, L)$ is the largest possible, i.e., equals $n$. If this is the case,
the splitting algebra $L$ for $V$ is said to be \emph{minimal} \cite[Definition 4.3]{AM} 
provided $L$ is the total quotient ring of the $D$-module $K$-subalgebra generated by all 
$f(V)$, $f \in \operatorname{Hom}_{K \# D} (V, L)$. 

\begin{theorem}[\textbf{Characterization}--\cite{AM}, Theorem~4.6]\label{thm:charac}
Let $K \subset L$ be an inclusion of AS $D$-module algebras, and assume $K^D = L^D$. Then
$L/K$ be a finitely generated PV extension if and only if $L$ is a minimal splitting algebra for
some $K\# D$-module with finite $K$-free rank.
\end{theorem}

Suppose that we are in the situation of the theorem above.
As the proof of the theorem shows, 
if $L$ is a minimal splitting algebra for $V$ as above, then 
$L^D$-linearly independent $n$ solutions 
$\boldsymbol{x}_j={}^t(x_{1j},\dots, x_{nj})$, $1 \le j \le n$, of \eqref{equation} in $L^n$ give an 
invertible $n \times n$ matrix $X = (x_{ij})$ with entries in $L$, and the $K$-subalgebra $A$ of $L$
generated by all entries $x_{ij}$ in $X$ together with $1/\det X$, turns into the principal $D$-module algebra for
the PV extension $L/K$.  Thus, in terminology of the standard Picard-Vessiot theories including
Hardouin's, $L/K$ is a \emph{Picard-Vessiot extension} for $V$ or for \eqref{equation}, with a \emph{fundamental
solution matrix} $X$ and a \emph{Picard-Vessiot ring} $A$. 
Conversely, every Picard-Vessiot extension or ring
in the standard sense arises in this manner.  

\begin{theorem}[\textbf{Unique existence}--\cite{AM}, Theorem~4.11]\label{thm:unique_exist}
Let $K$ be an AS $D$-module algebra. Assume that the field $K^D$ is algebraically closed. Then
for every left $K \# D$-module $V$ with finite $K$-free rank, there exists uniquely (up to
$D$-module algebra isomorphism over $K$) a minimal splitting algebra for $V$. 
\end{theorem}

\section{Preliminaries on $\times_R$-bialgebras}\label{sect:xR}

We continue to work over a fixed field $\Bbbk$. Let $R \neq 0$ be a commutative algebra.
By an $R$-\emph{ring} we mean an algebra given an algebra map from $R$. 
We recall from \cite{S3} the definition of $\times_R$-\emph{bialgebras} with some mild restriction.
Let $\mathcal{A}$ be an $R$-ring. Regard $\mathcal{A}$ as a left (resp., right) $R$-module by
the left (resp., right) multiplication by $R$. Let $\mathcal{A} \otimes_R \mathcal{A}$ denote the tensor product 
of the left $R$-module $\mathcal{A}$ with itself, and let 
\[ \mathcal{A} \times_R \mathcal{A} = \left\{\ \sum_i a_i \otimes b_i \in \mathcal{A} \otimes_R \mathcal{A} \biggm| 
\sum_i a_ix \otimes b_i = \sum_i a_i \otimes b_ix,\ \forall x \in R \right\}\ \]
denote the $R$-centralizers of the two right $R$-module structures on $\mathcal{A} \otimes_R \mathcal{A}$. Then 
$\mathcal{A} \times_R \mathcal{A}$ is naturally an $R$-ring \cite[p.~101]{S3} with respect to the product
\[ \Big(\sum_i a_i \otimes b_i\Big)\Big(\sum_j c_j \otimes d_j\Big) = \sum_{i,j} a_ic_j \otimes b_id_j \]
and the map $x \mapsto x \otimes 1 \, (=1 \otimes x)$ from $R$. 
As our mild restriction we pose the assumption that $\mathcal{A}$ is projective 
as a left $R$-module, which ensures the associativity
\begin{equation}\label{eq:assoc}
(\mathcal{A} \times_R \mathcal{A}) \times_R \mathcal{A} \simeq 
\mathcal{A} \times_R (\mathcal{A} \times_R \mathcal{A})
\end{equation} 
in the sense of \cite[Definition 2.6, p.~94]{S3}. Suppose that the left $R$-module $\mathcal{A}$ has an $R$-coalgebra
structure $\Delta : \mathcal{A} \to \mathcal{A} \otimes_R \mathcal{A}$, $\varepsilon : \mathcal{A} \to R$. 

\begin{definition}[Sweedler \cite{S3}]\label{def:xR}  
We say that $\mathcal{A}$ is a $\times_R$-\emph{bialgebra} if the following
are satisfied.
\begin{itemize}
\item[(a)] $\Delta( \mathcal{A} ) \subset \mathcal{A} \times_R \mathcal{A}$;
\item[(b)] $\Delta : \mathcal{A} \to \mathcal{A} \times_R \mathcal{A}$ is an algebra map;
\item[(c)] $\varepsilon(1) = 1$;
\item[(d)] $\varepsilon(ab) = \varepsilon(a \varepsilon(b))$ for all $a, b \in \mathcal{A}$.
\end{itemize}
\end{definition}

Let $\mathcal{A}$ be a $\times_R$-bialgebra. Given left $\mathcal{A}$-modules $M$, $N$, the action 
$$ a \rightharpoonup (m \otimes n) = \sum (a_1 \rightharpoonup m) \otimes (a_2 \rightharpoonup n),\quad
a \in \mathcal{A}, m \in M, n \in N$$
by $\mathcal{A}$ on the tensor product $M \otimes_R N$ over $R$ is well defined by Condition (a) above. 
By (b), $M \otimes_R N$
is a left $\mathcal{A}$-module with respect to this action. It follows by (c), (d) that $R$ is a left $\mathcal{A}$-module by
$$ a \rightharpoonup x = \varepsilon(ax), \quad a \in \mathcal{A}, x \in R.$$
Notice that the corresponding representation 
\begin{equation}\label{eq:alpha}
\alpha : \mathcal{A} \to \mathrm{End}(R), \ \alpha(a)(x) = a \rightharpoonup x
\end{equation}
coincides with the $\mathcal{I}$ map in \cite{S3}.

\begin{proposition}\label{prop:tensor}
The left $\mathcal{A}$-modules form a tensor category, $\mathcal{A} \text{-} \mathrm{Mod}$,
with respect the tensor product $M \otimes_R N$, the unit object $R$ as above, and
the obvious associativity and unit-constraints. If $\mathcal{A}$ is cocommutative as an $R$-coalgebra,
this tensor category is symmetric with respect to the obvious symmetry.
\end{proposition}

\begin{proof} 
This follows if one notices that the associativity \eqref{eq:assoc} ensures
that the obvious $R$-linear isomorphism $(M \otimes_R N) \otimes_R P \simeq M \otimes_R (N \otimes_R P) $
is $\mathcal{A}$-linear.
\end{proof}

To give examples of $\times_R$-bialgebras, let $H$ be a bialgebra (over $\Bbbk$), and let 
$R$ be an $H$-module algebra; this last does and will mean a left $H$-module algebra 
which is non-zero and commutative, as before.  
Then we have the smash-product algebra $R \# H$  
\cite[p.~153]{S2}. Regard this $R \# H$ as an $R$-ring by the natural embedding 
$R = R \otimes \Bbbk \to R \# H$, and as an $R$-coalgebra by base extension of 
the $\Bbbk$-coalgebra $H$ along $\Bbbk \to R$.

\begin{lemma}\label{lem:main-example}
Let $I \subset R \# H$ be an ideal and coideal, and set $\mathcal{A} = R \# H / \, I$; this is an $R$-ring and
$R$-coalgebra. If $\mathcal{A}$ is projective as a left $R$-module and is cocommutative as an $R$-coalgebra, 
then it is a $\times_R$-bialgebra.
\end{lemma}

\begin{proof}
The cocommutativity assumption implies that if $h \in H$, $x \in R$, then
\begin{equation*} 
(1\# h)x = \sum\varepsilon(h_1x)\# h_2 \equiv \sum\varepsilon(h_2x)\# h_1 
= \sum (h_2 \rightharpoonup x) \# h_1
\end{equation*}
modulo $I$ in $R\# H$. It follows that  
\begin{align*} 
&\sum (1 \# h_1)x \otimes (1 \# h_2) \equiv \sum ((h_2\rightharpoonup x) \# h_1) \otimes (1 \# h_3)\\ 
&= \sum (1 \# h_1) \otimes ((h_2\rightharpoonup x) \# h_3) 
= \sum (1 \# h_1) \otimes (1 \# h_2)x
\end{align*}
modulo $I \otimes (R \# H) + (R \# H)\otimes I$ in $(R \# H)\otimes (R \# H)$. 
This ensures Condition (a) of Definition \ref{def:xR}. The
remaining conditions are easily verified.
\end{proof}

\begin{corollary}[\cite{S3}, p.~117]\label{cor:smash-product}
 If the bialgebra $H$ is cocommutative, $R \# H$ is a $\times_R$-bialgebra.
\end{corollary}

The following example will be used in the proofs of Proposition \ref{prop:injective_representation}
and of Lemma \ref{lem:skew-der}. Suppose that $R\, (\ne 0)$ is a commutative algebra. 

\begin{example}\label{ex:End}
Regard $R^e := R \otimes_{\Bbbk} R$ as a (commutative) $R$-algebra
by $x \mapsto x \otimes 1, R \to R \otimes R$. The $R$-linear dual of $R^e$ is naturally
identified with the $R$-module $\mathrm{End}(R)$ consisting of all $\Bbbk$-linear endomorphisms on $R$.
Assume that $R$ is a field. Then we have the dual (cocommutative) $R$-coalgebra 
$(R^e)^{\circ}$ of $R^e$, which is now supposed to be included in $\mathrm{End}(R)$; 
see \cite[Section 6.0]{S2}. 
Note that $\mathrm{End}(R)$ is an algebra, and is, moreover, 
an $R$-ring, given the algebra map from $R$ which sends each $x \in R$
to the multiplication by $x$. It is known that the $R$-coalgebra $(R^e)^{\circ}$ is an $R$-subring
of $\mathrm{End}(R)$, and is indeed a cocommutative $\times_R$-bialgebra.
\end{example}

\section{Iterative $q$-difference rings and associated $\times_R$-bialgebras}\label{sect:iter-dif-ring}

\subsection{}\label{subsec:iterative_difference_ring}
We let $\mathbb{N} = \{ 0, 1, 2,... \}$ be the set of non-negative integers.

Throughout in this section, $C$ denotes a field, and $C(t)$ denotes the field of rational functions
over $C$. Choose arbitrarily an element $q \in C \smallsetminus \{ 0, 1 \}$. 
Let $\Bbbk_0$ denote the prime field included in $C$, and set $\Bbbk = \Bbbk_0(q)$, the subfield
of $C$ generated by $q$ over $\Bbbk_0$.
Following \cite{H} we denote the $q$-integer, the $q$-factorial and the $q$-binomial, 
respectively by
$$ [k]_q = \frac{q^k - 1}{q - 1}, \ [0]_q = 1,$$
$$ [k]_q ! = [k]_q [k - 1]_q ... [1]_q, \ [0]_q ! = 1,$$
$$ \binom{r}{s}_q = \frac{[r]_q !}{[s]_q ![r-s]_q !},$$
where $k, r, s \in \mathbb{N}$ with $k > 0,\ r \geq s$.

In what follows we fix a commutative ring $R$ including $C(t)$, and such an automorphism
$\sigma_q : R \overset{\simeq}{\longrightarrow} R$ that extends the $q$-difference operator 
$f(t) \mapsto f(qt)$ on $C(t)$. 

\begin{definition}[Hardouin~\cite{H}, Definition~2.4]\label{def:iter-dif-op}
An \emph{iterative} $q$-\emph{difference operator} on $R$ is a sequence
$\delta^*_R = (\delta^{(k)}_R)_{k \in \mathbb{N}}$ of maps $\delta^{(k)}_R : R \to R$
such that
\begin{itemize}
\item[(1)] $\delta^{(0)}_R = \mathrm{id}_R$, the identity map on $R$,
\item[(2)] $\delta^{(1)}_R = \dfrac{1}{(q-1)t}(\sigma_q - \mathrm{id}_R)$,
\item[(3)] $\delta^{(k)}_R( x + y ) = \delta^{(k)}_R(x) + \delta^{(k)}_R(y),\ x, y \in R$,
\item[(4)] $\delta^{(k)}_R( x y ) = \sum_{ i + j = k } \sigma_q^i \circ \delta^{(j)}_R(x) \, \delta^{(i)}_R( y ),\ x, y \in R$,
\item[(5)] $\delta^{(i)}_R \circ \delta^{(j)}_R = \dbinom{i+j}{i}_q \delta^{(i+j)}_R$. 
\end{itemize} 
\noindent
An \emph{iterative} $q$-\emph{difference ring} is a commutative ring $R \supset C(t)$ given $\sigma_q, \delta^*_R$ such as above. 
\end{definition}

Assume that $q$ is not a root of unity. Then, $[k]_q \neq 0$ for all $k$. If $\delta^*_R = (\delta^{(k)}_R)_{k \in \mathbb{N}}$
is an iterative $q$-difference operator on $R$, Conditions (1), (2) and (5) above require
\begin{equation}\label{eq:kth-power}
\delta^{(1)}_R = \frac{1}{(q-1)t}(\sigma_q - \mathrm{id}_R), \quad 
\delta^{(k)}_R = \frac{1}{[k]_q!}(\delta^{(1)}_R)^k, \ k \in \mathbb{N}.
\end{equation}
Conversely, if we define $\delta^{(k)}_R$ by \eqref{eq:kth-power}, then 
$\delta^*_R = (\delta^{(k)}_R)_{k \in \mathbb{N}}$ forms
an iterative $q$-difference operator on $R$; especially, Condition (4) is satisfied since one sees 
$\delta^{(1)}_R \circ \sigma_q = q \, \sigma_q \circ \delta^{(1)}_R$. Therefore under the assumption, 
an iterative $q$-difference ring is nothing but such a pair
$(R, \sigma_q)$ as above. In this case the results obtained by Hardouin \cite{H} are 
specialized from the difference Picard-Vessiot Theory as developed in \cite{PS1}. 

In what follows we assume that $q$ is a root of unity, and let $N\, (>1)$ denote its order, i.e., 
the least positive integer such that $q^N = 1$.  

\begin{lemma}\label{lem:relations}
Let $\delta^*_R = (\delta^{(k)}_R)_{k \in \mathbb{N}}$ be an iterative $q$-difference operator on $R$.
\begin{itemize}
\item[(a)] $\sigma_q^N = \mathrm{id}_R$.
\item[(b)] Each $\delta^{(k)}_R$ is $\Bbbk$-linear.
\item[(c)] $\delta^{(1)}_R(t) = 1$. 
\item[(d)] $\delta^{(k)}_R(t) = 0, \ 1 < k \in \mathbb{N}$. 
\item[(e)] \cite[Lemma 2.6]{H} $\delta^{(k)}_R \circ \sigma_q = q^k \, \sigma_q \circ \delta^{(k)}_R, \ k \in \mathbb{N}$.
\end{itemize}
\end{lemma}

\begin{proof}
(a) It follows from Conditions (2), (5) above that $(\delta^{(1)}_R)^N = 0$, and so
$\big(\frac{1}{t}(\sigma_q - \mathrm{id}_R)\big)^N = 0$. 
This implies the desired result, since we see by using 
$\big(\frac{1}{q^i}\sigma_q -\mathrm{id}_R\big)\frac{1}{t} 
= \frac{1}{t}\big(\frac{1}{q^{i+1}}\sigma_q-\mathrm{id}_R\big)$ that 
\[ \Big(\frac{1}{t}(\sigma_q - \mathrm{id}_R)\Big)^N = \frac{1}{t^N} \prod_{i=0}^{N-1}\Big(\frac{1}{q^i}\sigma_q - \mathrm{id}_R\Big) = \frac{1}{t^N}(\sigma^N_q-\mathrm{id}_R).  \]

(b) By using Condition (4), one sees by induction on $k$ 
\begin{equation}\label{delta-at-1}
\delta_R^{(k)}(1) = 0, \quad 0 < k \in \mathbb{N}. 
\end{equation}
By (3), this proves the desired result if $N =2$ or $q=-1$. So, we may suppose $N > 2$. 
Compute $\delta_R^{(k)} \circ \delta_R^{(1)} \circ \delta_R^{(1)}$ in two ways, using (5).
Then one sees by using \eqref{delta-at-1} that $\delta_R^{(k)}(q) \, \delta_R^{(2)} = 0$. 
Since $\delta^{(2)}_R(t^2)
= \dfrac{1}{[2]_q}\delta^{(1)}_R\circ \delta^{(1)}_R(t^2) = 1$, we have 
\begin{equation*}
\delta_R^{(k)}(q) = 0, \quad 0 < k \in \mathbb{N}. 
\end{equation*}
By (3) and (4), this together with \eqref{delta-at-1} implies the desired result. 

(c) This follows immediately from (2). 

(d) By Condition (5), the operators $\delta^{(k)}_R$, $k\in \mathbb{N}$, commute with each other. 

Let $k > 0$. One sees from \eqref{delta-at-1} and (c) above that
\begin{equation}\label{vanish_at_t} 
\delta^{(1)}_R \circ \delta^{(k)}_R(t) = \delta^{(k)}_R \circ \delta^{(1)}_R(t) = 0,
\end{equation}
which together with (2) proves
\begin{equation}\label{invariant_by_sigma}
\sigma_q\circ \delta^{(k)}_R(t) = \delta^{(k)}_R(t).
\end{equation}
By (2), we have  
\begin{equation}\label{two_condition_imply}
(q+1)\delta^{(k)}_R(t)= \delta^{(k)}_R \circ \delta^{(1)}_R(t^2) 
= \delta^{(1)}_R \circ\delta^{(k)}_R  (t^2). 
\end{equation}

We prove the desired equation by induction on $k > 1$. 
Suppose $k = 2$. By using \eqref{vanish_at_t}, \eqref{invariant_by_sigma} and Condition (4), 
one deduces from \eqref{two_condition_imply}
\[ (q+1)\delta^{(2)}_R(t)= \delta^{(1)}_R(2t\delta^{(2)}_R(t)+1) = 2\delta^{(2)}_R(t), \]
which implies the desired equation for $k = 2$, since $q \ne 1$. If $k > 2$, one deduces from 
\eqref{two_condition_imply} and the desired equations for smaller $k$ 
\[ (q+1)\delta^{(k)}_R(t)= \delta^{(1)}_R(2t\delta^{(k)}_R(t)) = 2\delta^{(k)}_R(t), \] 
which implies the desired equation for $k$. 
\end{proof}

It seems that (d) above is implicitly used in the proof of \cite[Lemma 2.6]{H}, to prove 
Eq.~(5) on Page 106, line 4; the lemma is reproduced above as (e). 

\begin{corollary}\label{cor:stable}
Let $\delta^*_R = (\delta^{(k)}_R)_{k \in \mathbb{N}}$ be an iterative $q$-difference operator on $R$.
Then each $\delta^{(k)}_R$ as well as $\sigma_q$ stabilizes $\Bbbk(t)$. The restricted operators 
$\delta^{(k)}_R|_{\Bbbk(t)}$, $k \in \mathbb{N}$, give the unique iterative $q$-difference operator on $\Bbbk(t)$,
as defined by Definition \ref{def:iter-dif-op} when $C = \Bbbk$, $R = \Bbbk(t)$, and they extend
the $\Bbbk$-linear operators on $\Bbbk[t]$ determined by
\begin{equation}\label{def_of_delta}
\delta^{(k)}_R(t^n) = 
\begin{cases}
\dbinom{n}{k}_q t^{n-k}, & n \ge k,\\ 
\ \, 0, & 0 \le n < k. 
\end{cases}
\end{equation}
\end{corollary}
\begin{proof}
The first assertion on stability follows 
from Lemma \ref{lem:relations} (b)--(d). By inductions first on $n$ and then on $k$, 
we see that Condition (4) uniquely determines the values $\delta^{(k)}_R(t^n)$ in $\Bbbk[t]$. By the same condition
the extension of the operators to $\Bbbk(t)$ is unique. As is essentially shown by \cite[Proposition 2.9]{H}, 
the $\Bbbk$-linear operators on $\Bbbk[t]$ defined by \eqref{def_of_delta} uniquely extend to 
an iterative $q$-difference operator on $\Bbbk(t)$. This proves the second assertion. 
\end{proof}

\begin{remark}\label{rem:main_example}
Let $\delta^*_R = (\delta^{(k)}_R)_{k \in \mathbb{N}}$ be an iterative $q$-difference operator on $R$.
Assume that
\begin{equation}\label{constant}
\delta^*_R\ \text{is constant on}\ C,\ \text{or}\ \delta^{(k)}_R(c) = 0,\ c \in C,\ k > 1. 
\end{equation}
Then one sees as proving the last corollary that $\delta^{(k)}_R$, $k \in \mathbb{N}$, stabilize $C(t)$, and
the restricted operators $\delta^{(k)}_R|_{C(t)}$, $k \in \mathbb{N}$, give the unique 
iterative $q$-difference operator on $C(t)$ that consists of $C$-linear operators.  
They extend the $C$-linear operators on $C[t]$ determined by \eqref{def_of_delta}, and  
coincide with those given by \cite[Proposition 2.10]{H} as a main
example of iterative $q$-difference operators. 
Since $\delta^*_R$ may not stabilize $C(t)$ in general, the tensor
products in \cite[Lemma 2.12, Proposition 2.12]{H} should be taken over $\Bbbk(t)$, not over $C(t)$,
as far as the authors understand. 
\end{remark}

\subsection{}\label{subsec:construction_of_H}
We are going to construct a Hopf algebra $H$ over $\Bbbk = \Bbbk_0(q)$ which is closely related 
with iterative $q$-difference operators. In what follows we suppose that $\Bbbk= \Bbbk_0(q)$ is our ground field,
and let $\otimes$ denote the tensor product over $\Bbbk$. 
Vector spaces and (Hopf) algebras mean those over $\Bbbk$.

Let $G = \langle \sigma \mid \sigma^N = 1 \rangle$ denote the cyclic group of order $N$ generated 
by an element $\sigma$.
As usual, the group algebra $\Bbbk G$ is regarded as a Hopf algebra with $\sigma$ grouplike, i.e., 
$\Delta(\sigma) = \sigma \otimes \sigma$, $\varepsilon(\sigma) = 1$.
Let 
$$ B = \bigoplus_{i = 0}^\infty \, \Bbbk \delta^{(i)}$$
denote a vector space with basis $\delta^{(i)}, i \in \mathbb{N}$. 

Recall the definition of braided tensor category $\mathcal{YD}^{\Bbbk G}_{\Bbbk G}$ of the Yetter-Drinfeld
modules over ${\Bbbk G}$; see \cite[Section 10.6]{Mo}.
The Yetter-Drinfeld modules with which we shall treat here are, as the opposite-sided version of those
defined by \cite[Definition 10.6.10]{Mo}, supposed to be right $\Bbbk G$-modules and right 
$\Bbbk G$-comodules. The following is directly verified.

\begin{lemma}\label{lem:Yetter-Drinfeld}
$B$ is an object in $\mathcal{YD}^{\Bbbk G}_{\Bbbk G}$ with respect to the structure
$$\delta^{(i)} \leftharpoonup \sigma = q^i \, \delta^{(i)}, $$
$$\delta^{(i)} \mapsto \delta^{(i)} \otimes \sigma^i, \ B \to B \otimes \Bbbk G, $$
where $i \in \mathbb{N}$. Moreover, $B$ is a braided Hopf algebra in $\mathcal{YD}^{\Bbbk G}_{\Bbbk G}$
with respect to the algebra structure
\begin{equation}\label{alg-struc}
\delta^{(i)} \, \delta^{(j)} = \binom{i+j}{i}_q \, \delta^{(i+j)}, \quad \delta^{(0)}=1, 
\end{equation}
and the coalgebra structure
$$ \Delta(\delta^{(k)}) = \sum_{i+j=k} \, \delta^{(i)} \otimes \delta^{(j)}, \quad \varepsilon(\delta^{(k)}) = \delta_{k,0}, $$
where $i, j, k \in \mathbb{N}$.
\end{lemma}

Radford's biproduct (or bozonization) \cite{R} constructs from $B$ a Hopf algebra, $\Bbbk G \star B$. Let
$$H = (\Bbbk G \star B)^{cop}$$
denote the Hopf algebra obtained from $\Bbbk G \star B$ by replacing the coproduct with its opposite.
We see easily the following.

\begin{proposition}\label{prop:Hopf-alg-H}
The Hopf algebra $H$ is characterized by the following properties:
\begin{itemize}
\item[(a)] $H$ includes $\Bbbk G$ as a Hopf subalgebra;
\item[(b)] $H = \bigoplus_{i=0}^{\infty} \, (\Bbbk G) \delta^{(i)}$, a free left $\Bbbk G$-module with
basis $\delta^{(i)}, i \in \mathbb{N}$;
\item[(c)] The algebra structure on $H$ is determined by \eqref{alg-struc} and
$$ \delta^{(i)} \, \sigma = q^i \, \sigma \, \delta^{(i)}, \quad i \in \mathbb{N}; $$
\item[(d)] The coalgebra structure on $H$ is determined by
$$ \Delta(\delta^{(k)}) = \sum_{i+j=k} \, \sigma^j \, \delta^{(i)} \otimes \delta^{(j)}, \quad 
\varepsilon(\delta^{(k)}) = \delta_{k,0}, \ k \in \mathbb{N}.$$
\end{itemize}
\end{proposition}

\begin{remark}\label{rem:grading}
As a braided Hopf algebra, $B = \bigoplus_{i=0}^{\infty} \, B(i)$ is strictly graded if we set
$B(i) = \Bbbk \delta^{(i)}$, $i \in \mathbb{N}$.  It follows that as a Hopf algebra, 
$H = \bigoplus_{i=0}^{\infty} \, H(i)$ is coradically graded if we set $H(i) = (\Bbbk G) \delta^{(i)}$,
$i \in \mathbb{N}$.
In particular, $H$ is a pointed Hopf algebra, which is neither commutative nor cocommutative, as is easily
seen.
\end{remark}

We set $\delta = \delta^{(1)}$. Since for every $ 0 \leq i < N$, \ $\delta^{(i)}$ is a multiple 
of $\delta^i$ by $[i]_q^{-1} \neq 0$, and $\delta^N = 0$, it follows that
$$ J := \bigoplus_{0 \leq i < N} \, (\Bbbk G) \delta^{(i)} \subset H $$
is a Hopf subalgebra which is generated by $G$, $\delta$.
Set 
$$ d_n = \delta^{(nN)}, \quad n \in \mathbb{N}. $$
Then it follows by $\dbinom{rN}{sN}_q = \dbinom{r}{s}$ (see \cite[Eq.~(1)]{H}) that
$$ d_m\, d_n= \binom{m+n}{m} \, d_{m+n}, \quad m, n \in \mathbb{N}. $$
Since for every $ 0 \leq i < N$, \ $\delta^{(nN+i)}$ is a multiple of $\delta^{(i)} d_n$
by $\dbinom{n+i}{i}_q^{-1} \neq 0$, one sees that $H$ is a free left (and right) $J$-module
with basis $d_n$, $n \in \mathbb{N}$. Thus,
$$ H = \bigoplus_{n=0}^{\infty} \, Jd_n. $$

\subsection{}\label{subsec:main_theorem}
Keep $R$, $\sigma_q$ as above. We see the following from Lemma \ref{lem:relations} and 
Proposition \ref{prop:Hopf-alg-H}.

\begin{proposition}\label{prop:corresp}
There is a one-to-one correspondence between
\begin{itemize}
\item the set of iterative $q$-difference operators $\delta_R^{*}$ on $R$, and
\item the set of left $H$-module structures $\rightharpoonup : H \otimes R \to R$
on $R$ such that
\begin{itemize}
\item[(i)] $h \rightharpoonup xy = \sum (h_1 \rightharpoonup x)(h_2 \rightharpoonup y), \
h \rightharpoonup 1 = \varepsilon (h)1, \ h \in H$, $x, y \in R$,
\item[(ii)] $\sigma \rightharpoonup x = \sigma_q(x), \ x \in R$, and
\item[(iii)] $\delta \rightharpoonup x = \dfrac{1}{(q-1)t} (\sigma_q(x) - x ), \ x \in R$.
\end{itemize}
\end{itemize}
Given a left $H$-module structure $\rightharpoonup$ from the second set, the corresponding 
iterative $q$-difference operator $\delta_R^{*} = (\delta^{(k)}_R)$ is given by 
$$\delta^{(k)}_R(x) = \delta^{(k)} \rightharpoonup x, \quad x \in R. $$
\end{proposition}

Assume that we are given a left $H$-module structure $\rightharpoonup : H \otimes R \to R$
on $R$ which satisfies Conditions (i)--(iii) above.
By (i), $R$ is an $H$-module algebra so that we can construct the smash-product
algebra $R \# H$; this is regarded as before, as an $R$-ring and $R$-coalgebra.
Given an element $h \in H$, we denote the element $1 \# h$ in $R \# H$ 
simply by $h$. Let
$$ I = \left( \delta - \frac{1}{(q-1)t}(\sigma - 1) \right)$$
denote the ideal of $R \# H$ generated by the one element, and define 
\begin{equation}\label{def_of_mathcalH}
\mathcal{H} = R \# H / \, I .
\end{equation} 
The semi-direct product $R \rtimes G \, (= R \# \Bbbk G)$ arises from the restricted action by
$G$ on $R$, and we have natural $R$-ring maps $R \rtimes G \hookrightarrow R \# H \to \mathcal{H}$.

\begin{lemma}\label{lem:coideal}
We have the following. 
\begin{itemize}
\item[(1)] $I$ is a coideal of $R \# H$, so that $\mathcal{H}$ is an $R$-coalgebra.
\item[(2)] The natural images of $d_n$, $n \in \mathbb{N}$, form a left $R \rtimes G$-free basis in $\mathcal{H}$.
\end{itemize}
\end{lemma}
\begin{proof}
Set 
$$ u :=  \frac{1}{(q-1)t} \, \in R, \quad \xi := \delta - u(\sigma - 1) \, \in R \# H. $$
Then $I$ is generated by $\xi$. Note that the coproduct $\Delta(\xi)$ on $R \# H$ is given by
\begin{equation}\label{Delta-of-xi}
\Delta(\xi) = \sigma \otimes \xi + \xi \otimes 1 \, .
\end{equation}
Note that $\sigma \rightharpoonup u = \frac{1}{q}u$, 
$\delta \rightharpoonup u = -\frac{q-1}{q}u^2$, and
$\xi \rightharpoonup x = 0$ for all $x \in R$. Then we compute in $R \# H$, 
\begin{equation}\label{relations}
\xi \, \sigma = q \, \sigma \, \xi \, , \quad \xi \, \delta = \delta \, \xi + \frac{q-1}{q}u \, \xi \, , \quad
\xi \, x = (\sigma \rightharpoonup x) \xi, \, x \in R. 
\end{equation}

With the action by $H$ restricted to $J$, $R$ is a $J$-module algebra. The associated
smash product $R \# J$ is an $R$-subring and $R$-subcoalgebra of $R \# H$
which contains $\xi$. Let $I_0$
denote the ideal of $R \# J$ generated by $\xi$. We see from \eqref{relations} that $I_0 = R \, \xi J$.
This together with \eqref{Delta-of-xi} proves that $I_0$ is a coideal. The 
natural embedding composed with the canonical projection 
\begin{equation}\label{isom-from-RG}
R \rtimes G \hookrightarrow R \# J \to R \# J /I_0
\end{equation}
is an isomorphism, since we see that the $R$-ring map
$R \# J \to R \rtimes G$ well defined by $\sigma \mapsto \sigma$, $\delta \mapsto u(\sigma - 1)$
induces an inverse.

We see that 
$$ I = \bigoplus_{n=0}^{\infty} \, I_0 \, d_n \, , $$
since the direct sum on the right-hand side is an ideal of $R \# H$, as is seen from the fact 
that each $d_n$ commutes with $\xi$. The last equation together with the isomorphism
given in \eqref{isom-from-RG} concludes the proof.
\end{proof}

We denote still by $d_n$ its natural image in $\mathcal{H}$. Then we have
\begin{equation}\label{RG-free}
\mathcal{H} = \bigoplus_{n=0}^{\infty} \, RG d_n = \bigoplus_{n=0}^{\infty} \, R d_nG
\end{equation} 
by Part 2 above, and since each $d_n$ commutes with $\sigma$.
Since by the right multiplication, $G$ acts on $\mathcal{H}$ as $R$-coalgebra automorphisms,
$\mathcal{H}$ is a module $R$-coalgebra over the group $R$-Hopf algebra $RG$. Here, we emphasize
that each element in $R \, (\subset RG)$ is supposed to act on $\mathcal{H}$ by the \emph{left} multiplication.
Let $(RG)^+$ denote the augmentation ideal of $RG$, i.e., the kernel of the counit 
$\varepsilon : RG \to R$, and set
$$ Z = \mathcal{H} / \mathcal{H}(RG)^+. $$
This is a quotient $R$-coalgebra of $\mathcal{H}$. Let $\pi : \mathcal{H} \to Z$ denote the
quotient map, and set $\overline{d}_n = \pi(d_n)$, $n \in \mathbb{N}$. 

\begin{lemma}\label{lem:R-free}
$Z$ is free as an $R$-module,
$$ Z = \bigoplus_{n=0}^{\infty} \, R \overline{d}_n , $$
with basis $\overline{d}_n , n \in \mathbb{N}$. The $R$-coalgebra structure on $Z$ is 
determined by
\begin{equation}\label{R-coalgZ}
\Delta(\overline{d}_n) = \sum_{l+m=n} \, \overline{d}_l \otimes \overline{d}_m, \quad 
\varepsilon(\overline{d}_n) = \delta_{n,0}.
\end{equation}
\end{lemma}

\begin{proof}
This is verified directly.
\end{proof}

Let
$$ \gamma : \mathcal{H} = \bigoplus_{n=0}^{\infty} \, R d_n G \to RG $$
denote the projection onto the 0-th component.

\begin{lemma}\label{lem:convolution-invertible}
$\gamma$ is $RG$-linear, and is invertible with respect to the convolution product.
\end{lemma}

\begin{proof}
Obviously, $\gamma$ is $RG$-linear. Note that for every $n \in \mathbb{N}$, 
\begin{equation}\label{coradical_filtration}
\mathcal{H}_n := \bigoplus_{k \leq n} \, Rd_{k}G
\end{equation}
is an $R$-subcoalgebra of $\mathcal{H}$,
and $\mathcal{H}$ is a union of all $\mathcal{H}_n$. To see that $\gamma$ is invertible,
it suffices to show that the restriction $\gamma |_{\mathcal{H}_n}$
is invertible in the $R$-algebra $\mathrm{Hom}_R(\mathcal{H}_n, RG)$ of all $R$-linear
maps $\mathcal{H}_n \to RG$. Since the $R$-coalgebra $\mathcal{H} = \bigcup_{n} \, \mathcal{H}_n$ 
is filtered so that $\Delta(\mathcal{H}_n) \subset \sum_{l + m = n} \, \mathcal{H}_l \otimes_R \mathcal{H}_m$
(see Remark \ref{rem:grading}), we see that the kernel of the restriction map 
$\mathrm{Hom}_R(\mathcal{H}_n, RG) \to \mathrm{Hom}_R(\mathcal{H}_0, RG)$
is a nilpotent ideal. The desired invertibility follows since $\gamma |_{\mathcal{H}_n}$ 
is restricted to the invertible, identity map on $\mathcal{H}_0 = RG$.
\end{proof}

By the dual result of \cite[Theorem 7.2.2]{Mo}, it follows from the lemma above that 
the right $RG$-module $R$-coalgebra $\mathcal{H}$ is isomorphic to the crossed
coproduct which is constructed on $Z \otimes RG$ by the coaction 
$$ \rho : Z \to RG \otimes_R Z, \quad \rho(\overline{d}_n) = (\gamma \otimes \pi) \circ \Delta(d_n) $$
and the (dual) cocycle
$$ \tau : Z \to RG \otimes_R RG, \quad \tau(\overline{d}_n) = (\gamma \otimes \gamma) \circ \Delta(d_n), $$ 
where $\Delta(d_n)$ denotes the coproduct on $\mathcal{H}$. We see from $\sigma^N = 1$ that $\rho$ is trivial, i.e.,
$\rho(\overline{d}_n) = 1 \otimes \overline{d}_n, \ n \in \mathbb{N}$; this is equivalent to 
saying that $\pi : \mathcal{H} \to Z$ is co-central. 
By definition each $\tau(\overline{d}_n)$ is contained in 
$\Bbbk (t)G \otimes_{\Bbbk (t)} \Bbbk (t)G$.
Let $G$ act on $\mathcal{H}$, $RG$ and $\mathcal{H} \otimes_R \mathcal{H}$ by the conjugations
$$ x \mapsto gxg^{-1},\quad x \otimes y \mapsto gxg^{-1} \otimes gyg^{-1}. $$
Since $\Delta$, $\gamma$ are $G$-equivariant (or preserve the conjugation), and each $\overline{d}_n$ is 
$G$-invariant,
it follows that $\tau(\overline{d}_n)$ is $G$-invariant.  If we denote the field of $G$-invariants
in $\Bbbk (t)$ by
$$ K = \Bbbk (t^N), $$
it follows that $\tau(\overline{d}_n) \in KG \otimes_K KG, \, n \in \mathbb{N}$.
Let 
$$ Z_K = \bigoplus_{n = 0}^{\infty} \, K \overline{d}_n $$
denote the obvious $K$-form of $Z$. Let 
$$ 1 \to \mathrm{Reg}(Z_K, K) \to \mathrm{Reg}(Z_K, KG) \to \mathrm{Reg}(Z_K, KG \otimes_K KG) \to... $$
denote the complex for computing the dual Sweedler cohomology (see \cite{S1}) of the $K$-Hopf algebra $KG$
with coefficients in the $KG$-comodule $K$-coalgebra $Z_K$, where the $KG$-comodule structure on $Z_K$ 
is trivial. 
Here, $\mathrm{Reg}(Z_K, A)$ denote the the abelian group of all convolution-invertible $K$-linear
maps $f : Z_K \to A$, where $A = (KG)^{\otimes n}$; it is identified with the group $A[[T]]^{\times}$
of all invertible power series over $A$, via $f \mapsto \sum_{n=0}^{\infty} \, f(\overline{d}_n)T^n$.
We may suppose that $\tau$ is a 2-cocycle in the last complex. Note that the restriction 
$\tau |_{K\overline{d}_0}$ is trivial. 
Since the $K$-Hopf algebra $KG$ is cosemisimple,
and the $K$-coalgebra $Z_K$ includes $K\overline{d}_0$ as a unique simple $K$-subcoalgebra,  
it follows by \cite[Theorem 4.1]{M} that $\tau$ is the coboundary $\partial \nu$ of 
some 1-cochain $\nu \in \mathrm{Reg}(Z_K, KG)$, i.e., 
$$ \tau(\overline{d}_n) = \sum_{k + l + m = n} (\nu(\overline{d}_k) \otimes 1) \Delta(\nu^{-1}(\overline{d}_l)) 
(1 \otimes \nu(\overline{d}_m)), \ n \in \mathbb{N}. $$
Since $(\varepsilon \otimes \varepsilon)\circ \tau = \varepsilon$, we have $\varepsilon\circ \nu = \varepsilon$. 
Since $\tau(\overline{d}_0) = 1 \otimes 1$, we have $\nu(\overline{d}_0) \in G$. By replacing $\nu$ with
$$ \overline{d}_n \mapsto \nu(\overline{d}_n)\, \nu(\overline{d}_0)^{-1},\quad n \in \mathbb{N}, $$
we may suppose $\nu(\overline{d}_0) = 1$. 
Define elements $d_n'$ in $\mathcal{H}$ by
$$ d_n' = \sum_{l + m = n} \nu(\overline{d}_l) \, d_m, \quad n \in \mathbb{N}. $$

\begin{proposition}\label{prop:cocom}
Keep the notation as above.
\begin{itemize}
\item[(1)] We have
\begin{equation*}
\mathcal{H} = \bigoplus_{n=0}^{\infty} \, R d_n' G = \bigoplus_{n=0}^{\infty} \, RG d_n'.
\end{equation*}
\item[(2)] The following hold in $\mathcal{H}$:
\begin{equation*}
d_0' = 1, \quad d_n' \sigma = \sigma d_n', \ n \in \mathbb{N}.
\end{equation*}
\item[(3)] We have 
\begin{equation*}
\Delta(d_n') = \sum_{l + m = n} d_l' \otimes d_m', \quad \varepsilon(d_n') = \delta_{n,0}, \ n \in \mathbb{N}
\end{equation*}
on $\mathcal{H}$. It follows that the $R$-coalgebra $\mathcal{H}$ is cocommutative.
\end{itemize}
\end{proposition}

\begin{proof}
The first equation of Part 2 holds since $\nu(\overline{d}_0) =1$, while the second holds
since $\sigma$ commutes with each element in $K$. 
The remaining parts follow from well-known results on crossed coproducts.
\end{proof} 

Recall that $R$ is naturally a left $R \# H$-module; see \eqref{module}. 
Since $I$ annihilates $R$, there is induced a left $\mathcal{H}$-module structure on $R$.

\begin{proposition}\label{prop:induced-structure}
We have the following.
\begin{itemize}
\item[(1)] The induced structure satisfies
\begin{equation*}
\sigma \rightharpoonup x = \sigma_q(x), \quad 
d_n' \rightharpoonup xy = \sum_{l + m = n} (d_l' \rightharpoonup x)(d_m' \rightharpoonup y), \ n \in \mathbb{N}
\end{equation*}
for all $x, y \in R$.
\item[(2)] The following relations hold in $\mathcal{H}$:
\begin{equation*}
\sigma x = (\sigma \rightharpoonup x) \sigma, \quad 
d_n' x = \sum_{l + m = n} (d_l' \rightharpoonup x) \, d_m'
\end{equation*}
for all  $x \in R, n \in \mathbb{N}$.
\end{itemize}
\end{proposition}

\begin{proof}
This follows easily from Part 3 of Proposition \ref{prop:cocom}.
\end{proof}

\begin{remark}\label{rem:product-in-d'}
By direct computations we see that the following hold in $\mathcal{H}$. 
\begin{itemize}
\item[(1)] For every $n \in \mathbb{N}$,
\begin{equation*}
\bigoplus_{k \leq n} Rd_{k}G = \bigoplus_{k \leq n} R d'_{k}G.
\end{equation*}
\item[(2)] For every $l, m \in \mathbb{N}$, 
\begin{equation*}
d_l' \, d_m' \equiv \binom{l + m}{l} \, d_{l+m}' \ \text{modulo} \bigoplus_{n < l+m} R d_n'G.
\end{equation*}
\end{itemize}
\end{remark}

Now, Lemma \ref{lem:main-example} and
Proposition \ref{prop:cocom}(3) prove the first assertion of the following key result of ours. 

\begin{theorem}\label{thm:H}
$\mathcal{H}$ is a cocommutative $\times_R$-bialgebra, whence the left $\mathcal{H}$-modules
$\mathcal{H}\text{-}\mathrm{Mod} = (\mathcal{H}\text{-}\mathrm{Mod}, \otimes_R, R)$
form a symmetric tensor category. An object in $\mathcal{H}\text{-}\mathrm{Mod}$ has its dual,
if it is finitely generated projective as an $R$-module. 
\end{theorem}

Note that the $\mathcal{H}$-module structure on $R$ given by \eqref{eq:alpha} coincides with
the action corresponding to the initially given $\sigma_q$ and iterative $q$-difference operator $\delta^*_R$ on $R$;
see the paragraph following Proposition \ref{prop:corresp}.  

To prove the remaining assertion on duality, we translate the argument by Hardouin \cite[p.119]{H}
into the language of $\times_R$-bialgebras, 
constructing a variation, $\Phi$, of the Ess map in \cite{S3};
see Lemma \ref{lem:for_duality}(3) below. Let $\mathcal{E} = \mathcal{H}\otimes_R \mathcal{H}$ denote the
tensor product of the right $R$-module $\mathcal{H}$ and the left $R$-module $\mathcal{H}$, in which we thus have
$ax \otimes b = a \otimes xb$, where $x \in R$, $a, b \in \mathcal{H}$. The vector space
\[ \mathcal{E}^R = \left\{\ \sum_i a_i \otimes b_i \in \mathcal{E} \biggm| 
\sum_i xa_i \otimes b_i = \sum_i a_i \otimes b_ix,\ \forall x \in R \right\}\ \]
of all $R$-centralizers in $\mathcal{E}$ forms an $R$-ring with respect to the product
\[ \Big(\sum_i a_i \otimes b_i\Big)\Big(\sum_j c_j \otimes d_j\Big) = \sum_{i,j} a_ic_j \otimes d_jb_i \]
and the map $x \mapsto x \otimes 1 \, (=1 \otimes x)$ from $R$.

Recall from \eqref{def_of_mathcalH} the Hopf algebra $H$ and the smash product
$R \# H$ which define $\mathcal{H}$. Let $\mathcal{S}$ denote the antipode of $H$.

\begin{lemma}\label{lem:for_duality}
We have the following.
\begin{itemize}
\item[(1)] Given $h \in H$, the natural image of $\sum(1\# h_1)\otimes (1\# \mathcal{S}(h_2))$ in $\mathcal{E}$
lies in $\mathcal{E}^R$. 
\item[(2)] $h \mapsto \sum(1\# h_1)\otimes (1\# \mathcal{S}(h_2))$ gives rise to an $R$-ring map
$R\# H \to \mathcal{E}^R$. 
\item[(3)] The map just obtained factors through $\mathcal{H}$, so that we have an $R$-ring map,
\[ \Phi : \mathcal{H} \to \mathcal{E}^R. \]
\end{itemize}
\end{lemma}
\begin{proof}
Let $h \in H$, $x \in R$. We write $h$ for $1\# h$, as before. Set $\varphi(h)= 
\sum h_1 \otimes \mathcal{S}(h_2)$ in $\mathcal{E}$. 

(1) By the cocommutativity of $\mathcal{H}$, we have 
$hx = \sum (h_2\rightharpoonup x)h_1$ in $\mathcal{H}$. Using this twice we have
\[ \varphi(h) x= \sum h_1(\mathcal{S}(h_2)\rightharpoonup x) \otimes S(h_3) 
= x \varphi(h)  \]
in $\mathcal{E}$, which proves Part 1. 

(2) One sees that $h \mapsto \varphi(h)$ gives an algebra map $H \to \mathcal{E}^R$. This extends to 
an $R$-ring map from $R \# H$, since we have
$\varphi(h)(x \otimes 1)=\sum ((h_1\rightharpoonup x) \otimes 1) \varphi(h_2)$ in $\mathcal{E}^R$. 

(3) Indeed, one sees that the $R$-ring map annihilates $\delta - \frac{1}{(q-1)t}(\sigma -1)$. 
\end{proof}

\begin{proof}[Proof of Theorem \ref{thm:H}]
Let $M, N \in \mathcal{H}\text{-}\mathrm{Mod}$. The $R$-module $\operatorname{Hom}_R(M,N)$
of all $R$-linear maps $M \to N$ turns into a left $\mathcal{E}^R$-module by defining 
\[ \Big( \sum_i a_i \otimes b_i\Big) 
\rightharpoonup f : m \mapsto \sum_i a_i\rightharpoonup f(b_i \rightharpoonup m), \]
where $\sum_i \, a_i \otimes b_i\in \mathcal{E}^R$, $f \in \operatorname{Hom}_R(M,N)$. This turns, moreover, into a
left $\mathcal{H}$-module through the $R$-ring map $\Phi$. Suppose that $M$ is finitely generated projective
as an $R$-module and $N = R$. Then the left $\mathcal{H}$-module $\operatorname{Hom}_R(M,R)$
together with the canonical evaluation and co-evaluation maps give a dual object of $M$, as is easily seen. 
\end{proof}

Hardouin \cite[Definition 3.1]{H} defines the notion of \emph{iterative q-difference modules}.

\begin{proposition}\label{prop:iterative-q-difference-module}
An \emph{iterative} $q$-\emph{difference module} over $R$, as defined by \cite[Definition 3.1]{H}, 
is precisely
such a left $\mathcal{H}$-module that is finitely generated free as an $R$-module. 
An extension $S \supset R$ of iterative $q$-difference rings is precisely a commutative algebra
object $S$ in $\mathcal{H}\text{-}\mathrm{Mod}$ such that the canonical map $R \to S$ is injective.
\end{proposition}
\begin{proof}
This is easy to see. We only remark that $\sigma_q^N$ acts on any iterative $q$-difference module as zero,
as is seen just as proving Lemma \ref{lem:relations}(a). 
\end{proof}

\begin{remark}\label{rem:preferable_definition}
Note that $C$ plays a very minor role, which we may, and we do in this remark, assume to be $\Bbbk$. 
Notice from Corollary \ref{cor:stable} that $\Bbbk(t)$ is uniquely an iterative $q$-difference ring. 
We denote the associated cocommutative $\times_{\Bbbk(t)}$-bialgebra by $\mathcal{H}_{\Bbbk(t)}$. 
The authors prefer to define an iterative $q$-difference ring to be a commutative algebra object
in $\mathcal{H}_{\Bbbk(t)}$-$\operatorname{Mod}$, though the original definition requires in addition the
object to include $C(t)$. 
\end{remark}

\subsection{}\label{subsec:injective_representation}
We return to the situation before the last remark. 
Recall that the iterative $q$-difference ring $R$ is naturally a left $\mathcal{H}$-module.

\begin{proposition}\label{prop:injective_representation}
The corresponding representation $\alpha : \mathcal{H} \to \operatorname{End}(R)$ is an injection,
and its image is generated by $\sigma_q$, $\delta^{(k)}_R$, $k \in \mathbb{N}$, and the multiplications
by all elements in $R$. 
\end{proposition}
\begin{proof}
The assertion on the image follows since the $R$-ring $\mathcal{H}$ is generated by $\sigma$, $\delta^{(k)}$, 
$k \in \mathbb{N}$.

To prove the injectivity of $\alpha$, 
we use $\mathcal{H}_{\Bbbk(t)}$ given in Remark \ref{rem:preferable_definition}. We see from
\eqref{RG-free} that $\mathcal{H} = R\otimes_{\Bbbk(t)} \mathcal{H}_{\Bbbk(t)}$.  
Moreover, the composite $\mathcal{H} \to \operatorname{End}(R) \to \operatorname{Hom}(k(t), R)$
of $\alpha$ with the restriction map coincides with the composite
$R\otimes_{\Bbbk(t)} \mathcal{H}_{\Bbbk(t)} \to R\otimes_{\Bbbk(t)} \operatorname{End}(\Bbbk(t))
\hookrightarrow \operatorname{Hom}(k(t), R)$ of the base extension of the representation
$\alpha_{\Bbbk(t)} : \mathcal{H}_{\Bbbk(t)} \to \operatorname{End}(\Bbbk(t))$ with the natural embedding. 
It follows that the desired injectivity will follows from the injectivity of $\alpha_{\Bbbk(t)}$. 
Therefore, by replacing $R$ with $\Bbbk(t)$, we may suppose that $R$ is a field. In that case we have
the cocommutative $\times_R$-bialgebra $(R^e)^{\circ}\, (\subset \operatorname{End}(R))$ given in Example
\ref{ex:End}. We see from \cite[Proposition 6.0.3]{S2} that $\sigma_q$ and
all $\delta^{(k)}_R$ are contained in $(R^e)^{\circ}$. Moreover, $\alpha$ gives an $R$-coalgebra
map $\mathcal{H} \to (R^e)^{\circ}$; this is indeed a $\times_R$-bialgebra map. 

Notice from Proposition \ref{prop:cocom} that the 
first term $\mathcal{H}_1$ of the coradical filtration \cite[p.185]{S2} on $\mathcal{H}$ is given by 
\begin{equation*}
\mathcal{H}_1 = RG \oplus Rd'_1G = \bigoplus_{0\le i < N}\mathcal{H}_1^{(i)}, 
\end{equation*}
where we have set $\mathcal{H}_1^{(i)} = R\sigma^i \oplus R d'_1\sigma^i$, $0\le i <N$; these 
are $R$-subcoalgebras of
$\mathcal{H}_1$.  (Moreover, by Remark \ref{rem:product-in-d'}(1), the $n$-th term coincides with the $\mathcal{H}_n$
given by \eqref{coradical_filtration}.)
By \cite[Theorem 5.3.1]{Mo}, the desired injectivity will follow
if we prove that the restriction $\alpha|_{\mathcal{H}_1} : \mathcal{H}_1 \to (R^e)^{\circ}$ is injective. 
This last injectivity is equivalent to 
\begin{itemize}
\item[(i)] for every $0\le i <N$, $\alpha|_{\mathcal{H}_1^{(i)}}$ is injective, and
\item[(ii)] if $i\ne j$, then $\alpha(\mathcal{H}_1^{(i)}) \cap \alpha(\mathcal{H}_1^{(j)}) = 0$.
\end{itemize}

To prove (i), we may suppose $i=0$, since the result in $i=0$ implies the result in the remaining cases,
as is easily seen. Since $\alpha(d'_1)$ is primitive, the desired result is equivalent to $\alpha(d'_1)\ne 0$;
this will follow by definition of $d'_1$, if one sees
$(\alpha(d_1)=)\, \delta^{(N)}_R \notin \sum_{0\le i <N} R\sigma_q^i$. 
On the contrary, suppose $\delta^{(N)}_R \in \sum_{0\le i <N} R\sigma_q^i$. Then 
$\delta^{(N)}_R(t^N) = \delta^{(N)}_R(1)t^N$, which contradicts \eqref{def_of_delta}. 

We see that (ii) holds, since if $i\ne j$, the coradicals $R\sigma_q^i$, $R\sigma_q^j$ 
of $\alpha(\mathcal{H}_1^{(i)})$, 
$\alpha(\mathcal{H}_1^{(j)})$ trivially intersect, or $\sigma_q^i\ne \sigma_q^j$.  
\end{proof}

\section{How cocommutative pointed Hopf algebras come in}\label{sec:proof_of_claim}

The main objective of this paper is to show the following. 

\begin{claim}\label{claim}
The main theorems of Hardouin \cite{H}, given below (with notation partially changed), follow from our 
results in Section \ref{sec:review} that are reproduced from \cite{AM}.
\end{claim}

\begin{theorem}[Hardouin \cite{H}, Theorem 4.7]\label{thm:HardouinThm4.7}
Let $K$ be an iterative $q$-difference field such that the field $C(K)$ of constants 
in $K$ is algebraically closed. Let $V$ be an iterative $q$-difference module over $K$. 
Then there exists an iterative $q$-difference Picard-Vessiot 
ring for $V$, which is unique up to isomorphism of iterative $q$-difference rings. 
\end{theorem}

\begin{theorem}[Hardouin \cite{H}, Theorem 4.12]\label{thm:HardouinThm4.12}
Let $K$ be as above, and set $C = C(K)$. 
Let $A$ be an iterative $q$-difference Picard-Vessiot ring over $K$. 
Then the group-valued functor $\underline{\operatorname{Aut}}(A/K)$, which associates to each commutative
$C$-algebra $T$, the group of all iterative $q$-difference $K \otimes_C T$-algebra automorphisms 
on $A \otimes_C T$,  
is an affine algebraic group scheme over $C$,
and is represented by the $C$-algebra $C(A \otimes_K A)$ 
of constants in $A \otimes_K A$ which is indeed finitely generated.  
\end{theorem}

\begin{theorem}[Hardouin \cite{H}, Theorem 4.20]\label{thm:HardouinThm4.20}
Let $K$, $A$ be as above. Let $\mathbf{G} = \underline{\operatorname{Aut}}(A/K)$ be the affine algebraic group
scheme as given above. Let $L$ denote the total quotient ring of $A$; then it uniquely turns into an
iterative $q$-difference ring extension of $A$, 
and is called the total iterative $q$-difference Picard-Vessiot extension of $A$. 
\begin{itemize}
\item[(1)] Given an intermediate total iterative $q$-difference ring $K \subset M \subset L$, $AM$ is
an iterative $q$-difference Picard-Vessiot ring over $M$. The correspondence $M \mapsto 
\underline{\operatorname{Aut}}(AM/M)$ gives an inclusion-reversing bijection from the set of all intermediate total
iterative $q$-difference rings $K \subset M \subset L$ to the set of all closed subgroup
schemes $\mathbf{H}$ in $\mathbf{G}$. The inverse is given by $\mathbf{H} \mapsto L^{\mathbf{H}}$,
where $L^{\mathbf{H}}$ consists of the $\mathbf{H}$-invariants in $L$ as defined in \cite[p.134, lines 18--20]{H}. 
\item[(2)] If $\mathbf{H} \subset \mathbf{G}$ is a normal closed subgroup scheme, then the $\mathbf{H}$-invariants
$A^{\mathbf{H}}$ in $A$ form an iterative $q$-difference Picard-Vessiot ring over $K$, and $L^{\mathbf{H}}$
is its total iterative $q$-difference Picard-Vessiot extension. 
Moreover, $\underline{\operatorname{Aut}}(A^{\mathbf{H}}/K)$ is naturally isomorphic 
to $\mathbf{G} \tilde{\tilde{/}} \mathbf{H}$. 
\item[(3)] Let $\mathbf{H} \subset \mathbf{G}$ is a closed subgroup scheme. 
Then the ring extension $L/L^{\mathbf{H}}$
is separable if and only if $\mathbf{H}$ is reduced. 
\end{itemize}
\end{theorem}

To prove Claim \ref{claim}, we work in the same situation as in the last section, using the same notation.
Thus, $R$ denotes an iterative $q$-difference ring, where $q$ is a root of unity of order $N\, (>1)$,
and $\mathcal{H}$ denotes the associated cocommutative $\times_R$-bialgebra.
Our main task is to present $R$ as a module algebra over an appropriate 
cocommutative pointed
Hopf algebra $D$ which satisfies the assumption \eqref{BW_assumption}.
Because this $D$ is distinct according to whether $\operatorname{char}\Bbbk$ is
zero or positive, we will discuss in the two cases, separately. 

\subsection{Case in characteristic zero}
Assume $\operatorname{char}\Bbbk = 0$, or in other words, $\Bbbk_0 = \mathbb{Q}$.
In this case we define a Hopf algebra $D$ over $\Bbbk$ by
$$ D = \Bbbk[d_1'] \otimes \Bbbk G. $$
This is the tensor product of the polynomial Hopf algebra $\Bbbk[d_1']$ and the group Hopf algebra 
$\Bbbk G$, where $d_1'$ is primitive, i.e., 
$\Delta(d_1') = 1 \otimes d_1' + d_1' \otimes 1, \varepsilon(d_1') =0$, and 
$ G = \langle \sigma \mid \sigma^N = 1 \rangle $, as before. 
This $D$ is a cocommutative pointed Hopf algebra, which necessarily 
satisfies the assumption \eqref{BW_assumption} since $\operatorname{char} \Bbbk = 0$.
The left $\mathcal{H}$-module action on $R$ given before, restricted to $\sigma$, $d_1'$, 
makes $R$ into a $D$-module algebra by Proposition \ref{prop:induced-structure}(1). 
The associated smash product $R \# D$ is a cocommutative $\times_R$-bialgebra by 
Corollary \ref{cor:smash-product}. 

\begin{proposition}\label{prop:isom}
The $R$-ring map $R \# D \to \mathcal{H}$ well defined by $\sigma \mapsto \sigma$, $d_1' \mapsto d_1'$
is an isomorphism of $\times_R$-bialgebras. It follows that $\mathcal{H} \text{-} \mathrm{Mod}$ coincides with
the symmetric tensor category $R \# D \text{-} \mathrm{Mod}$ of left $R \# D$-modules.
\end{proposition}

\begin{proof}
One sees from Proposition \ref{prop:cocom}(2), (3) and Proposition \ref{prop:induced-structure}(2) 
that the map is a well-defined $\times_R$-bialgebra map. It follows from Remark \ref{rem:product-in-d'} that
for every $n \in \mathbb{N}$,
$$ d_n' \equiv \frac{(d_1')^n}{n!} \ \text{modulo} \ \bigoplus_{k < n} Rd_k'G \, , $$
whence $\mathcal{H} = \bigoplus_{n=0}^{\infty} \, R(d_1')^nG$. This proves that the map is an isomorphism.
\end{proof}

Recall the results and the notation from Section \ref{sec:review}. 

\begin{proof}[Proof of Claim \ref{claim}~(in characteristic zero)]
Let $K$ be an \emph{iterative} $q$-\emph{difference field} \cite[p.107]{H}; this is the same as 
an AS $D$-module algebra that is connected, i.e., contains no non-trivial idempotent. 
Assume that the field $C(K)$ of \emph{constants} \cite[p.105]{H} (or equivalently, of $D$-invariants) in 
$K$ is algebraically closed. 
Let $V$ be an iterative $q$-difference 
module over $K$; this is the same as a left $K\# D$-module of finite $K$-dimension, see Proposition 
\ref{prop:iterative-q-difference-module}. By Theorem \ref{thm:unique_exist}, there exists uniquely
(up to isomorphism) a minimal splitting algebra $L$ for $V$. By Theorem \ref{thm:charac}, $L/K$ 
is a finitely generated 
PV extension of AS $D$-module algebras. As is seen from the paragraph following Theorem \ref{thm:charac}, 
the principal $D$-module algebra $A$ for $L/K$
is an \emph{iterative $q$-difference Picard-Vessiot ring} for $V$, as Hardouin \cite[p.128]{H} defines, 
and conversely, the total quotient ring of such a ring is a minimal splitting algebra for $V$ including the
ring as a principal $D$-module algebra. 
Such a ring is unique (up to isomorphism) as for $V$, as follows from our uniqueness of $L/K$ as for $V$, and 
of $A$ as for $L/K$. 
Thus obtained is Theorem \ref{thm:HardouinThm4.7}. 

The group-valued functor $\underline{\operatorname{Aut}}(A/K)$ given in Theorem \ref{thm:HardouinThm4.12}
is the same as our $\mathbf{Aut}_{D,K\text{-}\mathrm{alg}}(A)$, which we know is represented by 
the Hopf algebra $H$ for $L/K$; this $H$ is now finitely generated by Lemma \ref{lem:finite_generation}. 
Theorem \ref{thm:HardouinThm4.12} follows since $C(A \otimes_K A)$ coincides with our
$H =(A \otimes_K A)^D$. 

Our Theorem \ref{thm:Gal_corresp} together
with Remark \ref{rem:inverse} are now specialized to Parts 1, 2 of Theorem \ref{thm:HardouinThm4.20}.
(Part 1 refers to $``$an iterative $q$-difference Picard-Vessiot extension $AM$ over $M$," where $M$ may not
be a field. But Hardouin \cite[Definition 4.3]{H} defines the notion only over an iterative $q$-difference field.
Therefore, to justify the statement of Part 1, one has to re-define the notion over such an iterative $q$-difference
ring that is artinian and simple, just as we worked over an AS $D$-module
algebra to define the notion of principal $D$-module algebras.) 
Our Proposition \ref{prop:separable} is specialized to Part 3 of Theorem \ref{thm:HardouinThm4.20}. 
\end{proof}

\subsection{Case in positive characteristic} 
We start with proving the following for later use. 

\begin{proposition}[cf. \cite{H},\ Proposition 2.20]\label{prop:localization}
Suppose that $\operatorname{char}\Bbbk$ is arbitrary. 
Given a multiplicative set $S$ in $R$ such that
\begin{equation}\label{invertibility_assumption}
\text{for every} \ s \in S,\  \sigma_q\rightharpoonup s \ \text{is invertible in} \ S^{-1}R, 
\end{equation}
the localization $S^{-1}R$ uniquely turns into
an iterative $q$-difference ring so that $R \to S^{-1}R$ preserves the structure. 
\end{proposition}
\begin{proof}
The $R$-linear maps $\mathcal{H} \to S^{-1}R$ form an $R$-algebra $\operatorname{Hom}_R(\mathcal{H}, S^{-1}R)$
with respect to the convolution product. 
We see analogously to the proof of \cite[Lemma 2.7]{AM} that the $\Bbbk$-algebra map
\[ R \to \operatorname{Hom}_R(\mathcal{H}, S^{-1}R),\ x \mapsto [h \mapsto (h \rightharpoonup x)] \]
is localized by $S$, which proves the proposition. 
In fact, the image of each element of $S$ is invertible in $\operatorname{Hom}_R(\mathcal{H}, S^{-1}R)$, 
since its restriction to the $R$-subcoalgebra
$\mathcal{H}_0\subset \mathcal{H}$ is invertible by \eqref{invertibility_assumption}, 
and so is the restriction to each $\mathcal{H}_n$, $n \ge 0$; 
here, recall from the proof of Lemma \ref{lem:convolution-invertible},
the filtration $\mathcal{H} = \bigcup_{n} \, \mathcal{H}_n$, and note that the kernel of 
$\operatorname{Hom}_R(\mathcal{H}_n, S^{-1}R)\to \operatorname{Hom}_R(\mathcal{H}_0, S^{-1}R)$ is nilpotent,
here too.  
\end{proof}

Now, assume $\operatorname{char}\Bbbk = p > 0$, or $\Bbbk_0 = \mathbb{F}_p$. Let 
$$ Z_{\Bbbk} = \bigoplus_{n=0}^{\infty} \Bbbk d_n' $$
denote the obvious $\Bbbk$-form of the $R$-coalgebra $Z$; thus, 
$\Delta(d_n') = \sum_{l+m=n} \, d_l' \otimes d_m', \ \varepsilon(d_n') = \delta_{n,0}$ on $ Z_{\Bbbk}$.
The tensor algebra $T(Z_{\Bbbk})/(d_0' - 1)$ of $Z_{\Bbbk}$ divided by the relation $d_0' = 1$
has a unique Hopf algebra structure that extends the coalgebra structure on $Z_{\Bbbk}$.
We remark that the thus obtained Hopf algebra is Birkhoff-Witt; see \eqref{BW_assumption}.
For the Verschiebung map on it is surjective since the map on the coalgebra $Z_{\Bbbk}$, which is spanned by 
an $\infty$-divided power sequence, is so; see \cite[Page 504]{T}.
In the present case we define $D$ by the tensor product
\begin{equation}\label{D_positive_charac}
D = T(Z_{\Bbbk})/(d_0' - 1) \otimes \Bbbk G 
\end{equation}
of the two Hopf algebras, where $G = \langle \sigma \mid \sigma^N = 1 \rangle$, as before. 
This $D$ is a cocommutative pointed Hopf algebra which satisfies \eqref{BW_assumption}
by the remark above.
By extending uniquely the actions by $\sigma$, $d_n'$ in $\mathcal{H}$ on $R$, we can define 
a left $D$-module structure on $R$. With the thus defined structure, $R$ is in fact a
$D$-module algebra, and the associated smash product $R \# D$ is 
a cocommutative $\times_R$-bialgebra by Corollary \ref{cor:smash-product}. 

\begin{proposition}\label{prop:surj}
The $R$-ring map $R \# D \to \mathcal{H}$ well defined by $\sigma \mapsto \sigma$, $d_n' \mapsto d_n'$,
$0 < n \in \mathbb{N}$, is a surjection of $\times_R$-bialgebras.
It follows that $\mathcal{H} \text{-} \mathrm{Mod}$
is a tensor full subcategory of $R \# D \text{-} \mathrm{Mod}$. 
\end{proposition}
\begin{proof}
Similar to the first half of the proof of Proposition \ref{prop:isom}. The surjectivity follows by 
Proposition \ref{prop:cocom}(1). 
\end{proof}

\begin{proof}[Proof of Claim \ref{claim}~(in positive characteristic)]
We choose $D$ as above. Let $K$ be an iterative $q$-difference field with algebraically closed field of constants.
Choose this $K$ as the $R$ above.
Then $\mathcal{H}$ is a $\times_K$-bialgebra, and $K \in \mathcal{H} \text{-} \mathrm{Mod}$.  
By Proposition \ref{prop:surj}, the objects in $\mathcal{H} \text{-} \mathrm{Mod}$
are precisely those objects in $K \# D \text{-} \mathrm{Mod}$ 
which are annihilated by the kernel of $K \# D \to \mathcal{H}$. 
A $K \# D$-submodule of an object in $\mathcal{H} \text{-} \mathrm{Mod}$ is again in 
$\mathcal{H} \text{-} \mathrm{Mod}$. Therefore, working in $K \# D \text{-} \mathrm{Mod}$, 
we can discuss as in the proof in the zero-characteristic case,
except as for the existence of a minimal splitting algebra $L$ for $V$ over $K$, where $V \in 
\mathcal{H} \text{-} \mathrm{Mod}$ with $\dim_K V< \infty$. 
The question is whether 
the $L$, constructed in $ R \# D \text{-} \mathrm{Mod} $, is indeed 
in $ \mathcal{H} \text{-} \mathrm{Mod} $.

As is seen from the proof of \cite[Theorem 4.11]{AM}, $L$ is the 
total quotient ring of a simple $D$-module algebra $A$ including $K$ with the following
property: $A$ is a commutative $K$-algebra $K[x_{ij}]_{\det X}$ 
generated by all entries in $X = (x_{ij})$, and then localized at the determinant $\det X$, 
where $X$ is a square matrix which
is a solution of the equation, such as \eqref{equation}, associated with $V$. 
We see that $K[x_{ij}]$ is in $ \mathcal{H} \text{-} \mathrm{Mod} $. Since one sees that 
$\sigma_q \rightharpoonup \det X$ is a multiply of $\det X$ by some invertible element in $K$, it
follows by Proposition \ref{prop:localization} that $A$ and so $L$ are in $ \mathcal{H} \text{-} \mathrm{Mod} $,
as desired. 
\end{proof}

\subsection{}\label{subsec:added_remark} 
The proofs in the preceding two subsections are valid in the generalized situation that 
the iterative $q$-difference field $K$ over which everything is constructed
is replaced by a simple iterative $q$-difference ring which is artinian as a ring.
This generalization seems natural, amending the failure that given an 
intermediate iterated $q$-difference total ring $M$ in a Picard-Vessiot extension $L/K$, 
one cannot regard $L/M$ as a Picard-Vessiot extension.  

As further results on AS $D$-module algebras $K$, where $D$ is as in Section \ref{sec:review}, let us recall
the following: 
\begin{itemize}
\item 
Tannaka-Type Theorem \cite[Theorem 4.10]{AM}; it gives an equivalence of symmetric 
tensor categories, $\{\{ V\}\} \approx \mathbf{G} (L/K)\text{-}\mathrm{mod}$, between the abelian 
rigid tensor category $\{\{ V \}\}$ generated by a finite $K$-free object $V$ in 
$K \# D$-$\mathrm{Mod}$ and the category $\mathbf{G} (L/K)\text{-}\mathrm{mod}$
of finite-dimensional modules
over the PV group scheme $\mathbf{G} (L/K)$, 
where $L/K$ is a minimal splitting algebra for $V$; 
\item 
Solvability Criteria \cite[Theorem 2.10]{A}; it formulates and characterizes the solvability of such $V$ as above, 
in terms of $\mathbf{G} (L/K)$. 
\end{itemize} 

By the same observations as in the preceding two subsections, we see that these results hold
true for finitely generated PV extensions $L/K$ of simple iterative $q$-difference rings 
which are artinian. For that variation of the Tannaka-Type Theorem in positive characteristic, 
one should notice from the uniqueness of dual objects that the dual 
object of $V$ constructed in $K \# D\text{-}\mathrm{Mod}$, where $D$ is as in \eqref{D_positive_charac}, 
is necessarily an iterative $q$-difference module dual to $V$, if $V$ is such a module.

\section{Generalization}\label{sec:remarks}

Suppose that $\Bbbk$ is an arbitrary field over which we will work, and let $R\ne 0$ be a 
commutative algebra (over $\Bbbk$).
Fix an element $q \in \Bbbk \backslash 0$, and an endomorphism $\sigma : R \to R$. 
Heiderich \cite{He} defines a notion which generalizes iterative $q$-difference operators, as follows. 

\begin{definition}[Heiderich \cite{He}, Definition 2.3.13]\label{def:skew-iterative-derivation}
A \emph{$q$-skew iterative $\sigma$-derivation} on $R$ is a sequence
$\delta^*_R = (\delta^{(k)}_R)_{k \in \mathbb{N}}$ of maps $\delta^{(k)}_R : R \to R$
such that
\begin{itemize}
\item[(1)] $\delta^{(0)}_R = \mathrm{id}_R$, 
\item[(2)] $\delta^{(k)}_R \circ \sigma = q^k \, \sigma \circ \delta^{(k)}_R, \ k \in \mathbb{N}$, 
\item[(3)] each $\delta^{(k)}_R$ is $\Bbbk$-linear,
\item[(4)] $\delta^{(k)}_R( x y ) = \sum_{ i + j = k } \sigma^i \circ \delta^{(j)}_R(x) \, \delta^{(i)}_R( y ),\ x, y \in R$,
\item[(5)] $\delta^{(i)}_R \circ \delta^{(j)}_R = \dbinom{i+j}{i}_q \, \delta^{(i+j)}_R$.
\end{itemize}
\end{definition}

Let $\delta^*_R = (\delta^{(k)}_R)_{k \in \mathbb{N}}$ be such as defined above. A main difference 
from iterative $q$-difference operators is that one does not assume that $\delta^{(1)}_R$
is a multiple of $\sigma - \mathrm{id}_R$ by some element in $R$; cf. Condition (2) in Definition \ref{def:iter-dif-op}.
But we have the following, which is probably known.

\begin{lemma}\label{lem:skew-der}
Assume $\sigma \ne \mathrm{id}_R$ and that $R$ is a field.
Then there exists an element $u \in R$ such that
$$ \delta^{(1)}_R = u \, (\sigma - \mathrm{id}_R). $$
\end{lemma}

\begin{proof}
Here is a coalgebraic proof, which is hopefully new. Recall from Example \ref{ex:End} the cocommutative
$\times_R$-bialgebra $R^e\, (\subset \operatorname{End}(R))$. 
One sees from \cite[Proposition 6.0.3]{S2} that $\mathrm{id}_R$, $\sigma$,
$\delta^{(1)}_R$ (and all $\delta^{(k)}_R, k > 1$) are contained in $(R^e)^{\circ}$. Moreover,
$\mathrm{id}_R$ and $\sigma$ are distinct grouplikes, and $\delta^{(1)}_R$ is $(\mathrm{id}_R, \sigma)$-primitive, 
i.e., $\Delta(\delta^{(1)}_R) = \sigma \otimes \delta^{(1)}_R + \delta^{(1)}_R \otimes \mathrm{id}_R$.
Since $(R^e)^{\circ}$ is cocommutative, every $(\mathrm{id}_R, \sigma)$-primitive element is
necessarily of the form $u \, (\sigma - \mathrm{id}_R)$ for some $u \in R$.
\end{proof}

Keep $\sigma$, $\delta^*_R = (\delta^{(k)}_R)_{k \in \mathbb{N}}$ as above. 
Assume $\sigma \ne \mathrm{id}_R$ and that $R$ is a field. Let $u \in R$ be as in the last lemma. 

Either if $q=1$ and $\operatorname{char}\Bbbk = 0$, or if $q$ is not a root of unity, then just as in 
Section \ref{sect:iter-dif-ring}, 
$\delta^{(k)}_R$ are all determined by $\sigma$ and $u$; they are included 
in the $R$-subring generated by $\sigma$, and may be ignored. 

Assume $\delta^{(1)}_R \ne 0$, $\sigma \ne \sigma \circ \sigma$  and that 
$q$ is a root of unity of order $N > 1$.  Then Condition (2) in $k = 1$ implies that $\sigma(u) = \frac{1}{q}u$. 
Since Condition (5) implies $(\delta^{(1)}_R)^N =0$, one sees as proving Lemma \ref{lem:relations}(a)
that $\sigma^N = \mathrm{id}_R$.  
We can construct a 
cocommutative $\times_R$-bialgebra $\mathcal{H}$, just as in Section \ref{sect:iter-dif-ring}. 
We can also generalize
the argument in Section \ref{sec:proof_of_claim} to this $\mathcal{H}$. 
Consequently, all the results on iterative $q$-difference rings and modules that we have referred to in the section
are generalized to commutative algebra objects and their module objects in $\mathcal{H} \text{-} \mathrm{Mod}$. 

\section*{Acknowledgments}
The work was supported by
Grant-in-Aid for Scientific Research (C) 23540039, Japan Society of the Promotion of Science. 
The main results of this paper were announced by the first-named author at the 
AMS--AWM Special Session $``$Hopf Algebras and Their Representations" 
of Joint Mathematics Meetings, New Orleans, January 6--9, 2011. He thanks the 
organizers of the special session, Susan Montgomery, Richard Ng and Sarah Witherspoon.

\end{document}